\documentclass[11pt]{amsart}
\usepackage{amssymb, latexsym}
\theoremstyle{plain}
\newtheorem{theorem}{Theorem}
\newtheorem{corollary}{Corollary}

\newtheorem {lemma}{Lemma}
\newtheorem{proposition}{Proposition}

\numberwithin{equation}{section}
\renewcommand{\baselinestretch}{1.3}

\begin{document}

\title[Permutations induced by a random parking function ] {The distribution on permutations induced by a random parking function}

\author{Ross G. Pinsky}


\address{Department of Mathematics\\
Technion---Israel Institute of Technology\\
Haifa, 32000\\ Israel}
\email{ pinsky@technion.ac.il}

\urladdr{https://pinsky.net.technion.ac.il/}

\date{}

\begin{abstract}
A parking function on $[n]$ creates a permutation in $S_n$ via the order in which the $n$ cars appear in the $n$ parking spaces.
Placing the uniform probability measure on the set of parking functions on $[n]$ induces a probability measure on $S_n$. We initiate a study of some properties of this distribution.
Let $P_n^{\text{park}}$ denote this distribution on $S_n$ and let $P_n$ denote the uniform distribution on $S_n$. In particular, we obtain an explicit formula for
$P_n^{\text{park}}(\sigma)$ for all $\sigma\in S_n$. Then we show that for all but an asymptotically $P_n$-negligible set of permutations, one has
$P_n^{\text{park}}(\sigma)\in\left(\frac{(2-\epsilon)^n}{(n+1)^{n-1}},\frac{(2+\epsilon)^n}{(n+1)^{n-1}}\right)$.
However, this accounts for only an exponentially small part of the $P_n^{\text{park}}$-probability.  We also obtain an explicit formula for
$P_n^{\text{park}}(\sigma^{-1}_{n-j+1}=i_1,\sigma^{-1}_{n-j+2}=i_2,\cdots, \sigma^{-1}_n=i_j)$, the probability that the last $j$ cars park in positions
$i_1,\cdots, i_j$ respectively, and  show that
the  $j$-dimensional random vector $(n+1-\sigma^{-1}_{n-j+l}, n+1-\sigma^{-1}_{n-j+2},\cdots, n+1-\sigma^{-1}_{n})$
under $P_n^{\text{park}}$ converges in distribution to a random vector
$(\sum_{r=1}^jX_r,\sum_{r=2}^j X_r,\cdots, X_{j-1}+X_j,X_j)$, where $\{X_r\}_{r=1}^j$ are IID with the Borel distribution.
We then show that in fact for $j_n=o(n^\frac16)$, the final $j_n$ cars will park in increasing order with probability approaching 1 as $n\to\infty$.
We also obtain an explicit formula for
the expected value of the left-to-right maximum statistic $X_n^{\text{LR-max}}$, which counts the total number of
 left-to-right maxima in a permutation,
 and show that
$E_n^{\text{park}}X_n^{\text{LR-max}}$ grows approximately on the order $n^\frac12$.

\end{abstract}

\maketitle
\section{Introduction and Statement of Results}\label{intro}
\renewcommand{\baselinestretch}{1.3}
Consider a row of  $n$ parking spaces on a one-way street. A line of $n$ cars, numbered from 1 to $n$, attempt to park, one at a time. The $i$th car's preferred space is spot number $\pi_i\in[n]$.
If this space is already taken, then car $i$ proceeds forward and parks in the first available space, if one exists. If the  car is unable to park, it exits the street.
A sequence $\pi=\{\pi_i\}_{i=1}^n$ is called a parking function on $[n]$ if all $n$ cars are able to park. It is easy to see that $\pi$ is a parking function if and only if
$|\{i:\pi_i\le j\}|\ge j$, for all $j\in[n]$.  Let $\mathcal{P}_n$ denote the set of parking functions. It is well-known that
$|\mathcal{P}_n|=(n+1)^{n-1}$. There are a number of proofs of this result; a particularly  elegant one due to Pollack can be found in
\cite{FR}.
There is a large literature on parking functions and their generalizations; see, for example, the survey \cite{Y}.

We can consider a random parking function by placing the uniform probability measure on $\mathcal{P}_n$. Denote this probability measure by $P^{\mathcal{P}_n}$. A study of random parking functions was initiated by Diaconis and Hicks in \cite{DH}.
Since each parking function yields a permutation $\sigma=\sigma_1\cdots\sigma_n\in S_n$, where $\sigma_j$ is the number of the car that parked in space $j$, it follows that
 a random parking function induces a distribution on the set $S_n$ of permutations of $[n]$.
In this paper we initiate a study of this distribution.

We will use the notation $P_n$ and $E_n$ to denote the uniform probability measure and the corresponding expectation on $S_n$. We will denote by $P_n^{\text{park}}$  the probability measure
on $S_n$ induced by a random parking function in $\mathcal{P}_n$. The
corresponding expectation will be  denoted by $E_n^{\text{park}}$.
To be more precise concerning the definition of the induced probability measure, define $T_n:\mathcal{P}_n\to S_n$ by $T_n(\pi)=\sigma$,
if when using the parking function $\pi$,  $\sigma_j$ is the number of the car that parked in space $j$, for $j\in[n]$.
For example, if $n=4$ and $\pi=2213\in\mathcal{P}_4$, then we have $T_4(2213)=3124\in S_4$. We define
\begin{equation}\label{defprob}
P_n^{\text{park}}(\sigma)=P^{\mathcal{P}_n}\left(T_n^{-1}(\{\sigma\})\right).
\end{equation}

For $1\le i\le n<\infty$ and $\sigma\in S_n$, define
\begin{equation}\label{littlel}
l_{n,i}(\sigma)=\max\{l\in[i]: \sigma_i=\max(\sigma_i,\sigma_{i-1},\cdots, \sigma_{i-l+1}\}.
\end{equation}
For $n\in\mathbb{N}$, define
\begin{equation*}
L_n(\sigma)=\prod_{i=1}^nl_{n,i}(\sigma),\ \sigma\in S_n.
\end{equation*}
For example, if $\sigma=379218645\in S_9$, then $l_{n,i}(\sigma)=1$, for $i\in\{1,4,5,7,8\}$, $l_{n,i}(\sigma)=2$, for $i\in\{2,9\}$ and $l_{n,i}(\sigma)=3$, for $i\in\{3,6\}$.
Thus, $L_9(\sigma)=1^52^23^2=36$.
At certain points in the paper, it will be convenient to use the following variant of $l_{n,i}$.
 Define
\begin{equation}\label{littleltilde}
\tilde l_{n,i}(\sigma)=l_{n,\sigma^{-1}_i}(\sigma).
\end{equation}
 For example, if $\sigma=379218645$, then $\tilde l_{n,i}(\sigma)=1$, for $i\in\{1,2,3,4,6\}$, $\tilde l_{n,i}=2$, for
$i\in\{5,7\}$ and $\tilde l_{n,i}=3$, for $i\in\{8,9\}$.
Of course,
\begin{equation}\label{Ln}
L_n(\sigma)=\prod_{i=1}^nl_{n,i}(\sigma)=\prod_{i=1}^n\tilde l_{n,i}(\sigma),\ \sigma\in S_n.
\end{equation}

We have the following theorem.
\begin{theorem}\label{ptwisedist}
\begin{equation}\label{ptwiseform}
P_n^{\text{\rm park}}(\sigma)=\frac{L_n(\sigma)}{(n+1)^{n-1}}, \ \sigma\in S_n.
\end{equation}
\end{theorem}
The following corollary of Theorem \ref{ptwisedist} is immediate, where the asymptotic behavior follows from Stirling's  formula,  $n!\sim n^ne^{-n}\sqrt{2\pi n}$.
\begin{corollary}\label{corstir}
The expected value of the  random variable $L_n=L_n(\sigma)$ on $(S_n,P_n)$ satisfies
\begin{equation}\label{stir}
E_nL_n=\frac1{n!}\sum_{\sigma\in S_n}L_n(\sigma)=\frac{(n+1)^{n-1}}{n!}\sim\frac{e^{n+1}}{\sqrt{2\pi}\thinspace n^\frac32}.
\end{equation}
\end{corollary}
Although $E_nL_n$ is around $e^n$, with high $P_n$-probability $L_n(\sigma)$ is near $2^n$.
We will prove the following weak convergence result  for $L_n$.
\begin{theorem}\label{llnthm}
For any $\epsilon\in(0,2]$, the random variable $L_n=L_n(\sigma)$ on $(S_n,P_n)$ satisfies
\begin{equation}\label{wlln}
\lim_{n\to\infty}P_n((2-\epsilon)^n\le L_n\le (2+\epsilon)^n))=1.
\end{equation}
\end{theorem}
The following corollary follows immediately from Theorems \ref{ptwisedist} and \ref{llnthm}.
\begin{corollary}\label{corP-Ppark}
For any $\epsilon\in(0,2]$, the random variable $P_n^{\text{\rm park}}=P_n^{\text{\rm park}}(\sigma)$ on $(S_n,P_n)$ satisfies
\begin{equation}\label{P-Ppark}
\lim_{n\to\infty}P_n\left(\bigg\{\sigma\in S_n:P_n^{\text{\rm park}}(\sigma)\in\left(\frac{(2-\epsilon)^n}{(n+1)^{n-1}},\frac{(2+\epsilon)^n}{(n+1)^{n-1}}\right)\bigg\}\right)=1.
\end{equation}
\end{corollary}
And the following corollary follows immediately from Theorem \ref{ptwisedist}, Corollary \ref{corstir} and Stirling's formula.
\begin{corollary}\label{corexpP-Ppark}
The expectation of the random variable $P_n^{\text{\rm park}}=P_n^{\text{\rm park}}(\sigma)$ on $(S_n,P_n)$ is given by
\begin{equation}\label{expP-Ppark}
E_nP_n^{\text{\rm park}}=\frac1{n!}\sim\frac{e^n}{\sqrt{2\pi} n^{n+\frac12}}.
\end{equation}
\end{corollary}
Comparing \eqref{P-Ppark} and \eqref{expP-Ppark},
we see that
 for all but an asymptotically $P_n$-negligible set of permutations in $S_n$, the $P_n^{\text{park}}$-probability of a permutation in $S_n$ is approximately $\frac{2^n}{(n+1)^{n-1}}$, but the
``average'' $P_n^{\text{park}}$-probability of a permutation in $S_n$ is exponentially larger, namely asymptotic to $\frac{e^n}{\sqrt{2\pi} n^{n+\frac12}}$.
By Stirling's formula, it also follows that there is a set $A_n\subset S_n$ with $P_n(A_n)\to1$ for which  $P_n^{\text{park}}(A_n)$ is around $(\frac2e)^n$.
There is an asymptotically    $P_n$-negligible set of permutations in $S_n$
each of whose elements has super-exponentially larger $P_n^{\text{park}}$-probability  than the   average probability,
and an asymptotically
$P_n$-negligible set of permutations in $S_n$
each of whose elements has
  exponentially smaller $P_n^{\text{park}}$-probability   than the average probability.
In particular, we have the following corollary.
\begin{corollary}\label{largesmall}
The maximum value of $P_n^{\text{\rm park}}=P_n^{\text{\rm park}}(\sigma)$ is equal to
 $\frac{n!}{(n+1)^{n-1}}\sim\frac{\sqrt{2\pi}\thinspace n^\frac32}{e^{n+1}}$ and is attained uniquely at  $\sigma=1\cdots n$.
The minimum value of  $P_n^{\text{\rm park}}$ is equal to $\frac1{(n+1)^{n-1}}$ and is attained uniquely at $\sigma=n\cdots 1$.
\end{corollary}
\begin{proof}
The function $L_n=L_n(\sigma), \sigma\in S_n$,
 attains its maximum value $n!$ uniquely at $\sigma=1\cdots n$
and attains its minimum value 1 uniquely at $\sigma=n\cdots 1$.
.
\end{proof}

Let $\sigma^{-1}$ denote the inverse permutation of $\sigma$. So $\sigma^{-1}_k=j$ if and only if $\sigma_j=k$.
In car parking language,$\sigma^{-1}_k=j$ means that car number $k$ parked in space number $j$.
From the definition of a parking function,  it is obvious that
$$
P_n^{\text{park}}(\sigma^{-1}_1=j)=P^{\mathcal{P}_n}(\pi_1=j),\ j\in[n].
$$
In \cite{DH}, the following asymptotic behavior was proven for  $\pi_1$ (or any $\pi_k$ by symmetry):
\begin{equation}\label{DHresult}
\begin{aligned}
&\text{For fixed}\ j, P^{\mathcal{P}_n}(\pi_1=j)\sim\frac{1+P(X\ge j)}n;\\
&\text{For fixed}\ j, P^{\mathcal{P}_n}(\pi_1=n-j)\sim\frac{P(X\le j+1)}n,\\
&\text{where}\ X\ \text{is a random variable satisfying}\ P(X=j)=e^{-j}\frac{j^{j-1}}{j!},\ j=1,2,\cdots.
\end{aligned}
\end{equation}
The distribution $\{e^{-j}\frac{j^{j-1}}{j!}\}_{j=1}^\infty$ is called the Borel distribution. It is not obvious that it is a distribution, that is, that it sums to 1. For more on this, see \cite[p.135]{DH}.
It follows that \eqref{DHresult} also holds with
$P^{\mathcal{P}_n}$ and $\pi_1$ replaced  respectively by $P_n^{\text{park}}$ and $\sigma^{-1}_1$.
It would be nice to obtain asymptotic results  for  $P_n^{\text{park}}(\sigma^{-1}_{j_n}=k_n)$, for general $j_n,k_n$.
We pursue this  direction  when $j_n$ is near $n$, that is for the last cars to park.




\begin{theorem}\label{lastcars}

\begin{equation}\label{lastcarsformula}
\begin{aligned}
&P_n^{\text{park}}(\sigma^{-1}_{n-j+1}=i_1,\sigma^{-1}_{n-j+2}=i_2,\cdots, \sigma^{-1}_n=i_j)=\\
&\frac{(n-j)!\left(\prod_{l=1}^j(k_l-k_{l-1})^{k_l-k_{l-1}-2}\right)(n-k_j+1)^{n-k_j-1}\prod_{l=1}^j\tilde l_{n,n-j+l}(\sigma)}
{(n+1)^{n-1}\left(\prod_{l=1}^j(k_l-k_{l-1}-1)!\right)(n-k_j)!},
\end{aligned}
\end{equation}
where $\{k_l\}_{l=1}^j$ is the increasing rearrangement of $\{i_l\}_{l=1}^j$, $k_0=0$ and $\tilde l_{n,\cdot}(\sigma)$ is as in \eqref{littleltilde}.
\end{theorem}
One can use Theorem \ref{lastcars} to make  explicit calculations. Here we treat the cases $j=1$ and $j=2$ in depth in Corollaries \ref{j=1}-\ref{j=2asymp}.
This will lead us to a result for general $j$ in Corollary \ref{genj} and for growing $j=j_n$ in Theorem \ref{jnlastcars}.

We begin with the case $j=1$.
\begin{corollary}\label{j=1}
\begin{equation}\label{j=1formula}
P_n^{\text{park}}(\sigma^{-1}_n=k)=\frac1nk^k\binom nk\frac{(n-k+1)^{n-k-1}}{(n+1)^{n-1}}, \ k\in[n].
\end{equation}
\end{corollary}
\begin{proof}
Immediate from \eqref{lastcarsformula}, noting that if $\sigma^{-1}_n=k$, then $\tilde l_{n,n}(\sigma)=k$.
\end{proof}
From Corollary \ref{j=1} we obtain the following asymptotic formulas.
\begin{corollary}\label{j=1asymp}
\noindent i. For fixed $m\in \mathbb{N}$,
\begin{equation}\label{j=1endformula}
\lim_{n\to\infty}P_n^{\text{park}}(\sigma^{-1}_n=n+1-m)=\frac{m^{m-1}e^{-m}}{m!}.
\end{equation}
Thus, the random variable $n+1-\sigma^{-1}_n$ under $P_n^{\text{park}}$ converges in distribution to a random variable $X$ with the Borel distribution.


\noindent ii. For fixed $k\in\mathbb{N}$,
\begin{equation}\label{j=1beginformulaasymp}
P_n^{\text{park}}(\sigma^{-1}_n=k)\sim\frac{k^ke^{-k}}{k!}\frac1n,\ \text{as}\ n\to\infty.
\end{equation}

\noindent iii. Let $c_nn$ be an integer with $\lim_{n\to\infty}c_n=c\in(0,1)$. Then
\begin{equation}\label{bulk}
P_n^{\text{park}}(\sigma^{-1}_n=c_nn)\sim\frac1{(2\pi c)^\frac12(1-c)^\frac32}\frac1{n^\frac32}.
\end{equation}
\end{corollary}
\begin{proof}
The proof follows from \eqref{j=1formula} and standard asymptotic analysis. We write out the proof of  part (iii).
Substituting in \eqref{j=1formula}, we have
\begin{equation}\label{subfrom112}
P_n^{\text{park}}(\sigma^{-1}_n=c_nn)=\frac1n(c_nn)^{c_nn}\frac{n!}{(c_nn)!((1-c_n)n)!}\frac{(n-c_nn+1)^{n-c_nn-1}}{(n+1)^{n-1}}.
\end{equation}
Replacing the three factorials on the right hand side of \eqref{subfrom112} by their Stirling's formula approximations, $m!\sim m^me^{-m}\sqrt{2\pi m}$ as $m\to\infty$, and performing many cancelations,
we have
$$
\begin{aligned}
&\frac1n(c_nn)^{c_nn}\frac{n!}{(c_nn)!((1-c_n)n)!}\frac{(n-c_nn+1)^{n-c_nn-1}}{(n+1)^{n-1}}\sim\\
&\frac{n^{n-1}\left(1+\frac1{n-c_nn}\right)^{n-c_nn-1}}
{n\sqrt{c_n}(1-c_n)\sqrt{2\pi(1-c_n)n}\thinspace(n+1)^{n-1}}.
\end{aligned}
$$
Since $\lim_{n\to\infty}\frac{n^{n-1}\left(1+\frac1{n-c_nn}\right)^{n-c_nn-1}}{(n+1)^{n-1}}=1$, we conclude that
the right hand side of \eqref{subfrom112} is asymptotic to
$\frac1{(2\pi c)^\frac12(1-c)^\frac32}\frac1{n^\frac32}$.
\end{proof}

\noindent \bf Remark.\rm\
One can check that $\frac{k^ke^{-k}}{k!}$
is decreasing in $k$. Thus, from part (ii), for fixed $k$,  $P_n^{\text{park}}(\sigma^{-1}_n=k)$
is on the order $\frac1n$ and decreasing in $k$. From part (iii), the probability of $\sigma^{-1}_n$ taking on any particular value in the bulk is even smaller, namely on the order $n^{-\frac32}$.
And from part (i),
the distance of $\sigma^{-1}_n$ from $n$ converges in distribution as $n\to\infty$.

\medskip

We now turn to the case $j=2$.
\begin{corollary}\label{j=2}
Let
\begin{equation}\label{A}
A_n(a,b)=\frac{(n-2)!a^{a-2}(b-a)^{b-a-2}(n-b+1)^{n-b-1}}{(a-1)!(b-a-1)!(n-b)!(n+1)^{n-1}}, \ \text{for}\ 1\le a<b\le n.
\end{equation}
Then
\begin{equation}\label{j=2formula}
P_n^{\text{park}}(\sigma^{-1}_{n-1}=l,\sigma^{-1}_n=m)=\begin{cases}lmA_n(l,m),\ \text{if}\ 1\le l<m\le n;\\ (l-m)mA_n(m,l),\ \text{if}\ 1\le m<l\le n.\end{cases}.
\end{equation}
\end{corollary}
\begin{proof}
Immediate from \eqref{lastcarsformula}, noting that if $\sigma^{-1}_{n-1}=l$ and $\sigma^{-1}_n=m$, then
$$
\begin{cases}\tilde l_{n,n-1}(\sigma)=l, \ \tilde l_{n,n}(\sigma)=m,\ \text{if}\ l<m;\\ \tilde l_{n,n-1}(\sigma)=l-m,\ \tilde l_{n,n}(\sigma)=m,\ \text{if}\ l>m.\end{cases}
$$
\end{proof}
From Corollary \ref{j=2} we obtain the following
asymptotic formulas.
\begin{corollary}\label{j=2asymp}
\noindent i.
For fixed $l,m$,
\begin{equation}\label{j=2endasympformula}
\lim_{n\to\infty}P_n^{\text{park}}(\sigma^{-1}_{n-1}=n+1-l,\sigma^{-1}_n=n+1-m)=\frac{(l-m)^{l-m-1}m^{m-1}}{(l-m)!m!}e^{-l}, \ 1\le m<l;
\end{equation}
\begin{equation}\label{j=2endasympformulalower}
P_n^{\text{park}}(\sigma^{-1}_{n-1}=n+1-l,\sigma^{-1}_n=n+1-m)\sim\frac{(m-l)^{m-l}l^{l-1}}{l!(m-l)!}\frac1n,\ 1\le l<m.
\end{equation}
Thus, the random vector $(n+1-\sigma^{-1}_{n-1},n+1-\sigma^{-1}_n)$ under $P_n^{\text{park}}$ converges in distribution to a random vector
$(X_1+X_2,X_2)$
where $X_1$ and $X_2$ are IID with the Borel distribution.


\noindent ii. For fixed, $l,m$,
\begin{equation}\label{j=2beginasymptformula}
P_n^{\text{park}}(\sigma^{-1}_{n-1}=l,\sigma^{-1}_n=m)\sim\begin{cases}\frac{l^l}{l!}\frac{(m-l)^{m-l-1}m}{(m-l)!}e^{-m}\frac1{n^2},\ 1\le l<m;\\
\frac{m^m}{m!}\frac{(l-m)^{l-m}}{(l-m)!}e^{-l}\frac1{n^2},\ 1\le m<l.\end{cases}
\end{equation}
\end{corollary}
\begin{proof}
Standard asymptotic analysis.
\end{proof}
Part (i) of Corollaries \ref{j=1asymp} and \ref{j=2asymp} lead us to the following result for all $j$, which we will prove.
\begin{corollary}\label{genj}
Let $j\in\mathbb{N}$. The  $j$-dimensional random vector \newline$(n+1-\sigma^{-1}_{n-j+l}, n+1-\sigma^{-1}_{n-j+2},\cdots, n+1-\sigma^{-1}_{n})$
under $P_n^{\text{park}}$ converges in distribution to a random vector
$(\sum_{r=1}^jX_r,\sum_{r=2}^j X_r,\cdots, X_{j-1}+X_j,X_j)$, where $\{X_r\}_{r=1}^j$ are IID with the Borel distribution.

\noindent In particular then, for any $j\in\mathbb{N}$, the last $j$ cars from among the $n$ cars will park in increasing order
 with probability approaching 1 as $n\to\infty$:
 \begin{equation}\label{jinincorder}
\lim_{n\to\infty}P_n^{\text{park}}(\sigma^{-1}_{n+1-j}<\sigma^{-1}_{n+1-j+2}<\cdots<\sigma^{-1}_n)=1.
 \end{equation}
\end{corollary}
In fact, we can extend \eqref{jinincorder} to the last $j_n$ cars, where $j_n=o(n^\frac16)$. We will prove the following result.
\begin{theorem}\label{jnlastcars}
For $j_n=o(n^\frac16)$, the last $j_n$ cars from among the $n$ cars will park in increasing order with probability approaching 1 as $n\to\infty$:
\begin{equation}\label{jninincorder}
\lim_{n\to\infty}P_n^{\text{park}}(\sigma^{-1}_{n+1-j_n}<\sigma^{-1}_{n+1-j_n+1}<\cdots<\sigma^{-1}_n)=1.
\end{equation}
\end{theorem}

We now consider the left-to-right maximum statistic. Recall that a position $i\in[n]$ is called a \it left-to-right-maximum \rm for the permutation $\sigma\in S_n$
if $\sigma_i>\sigma_k$, for all $k\in[i-1]$. Let $X_n^{\text{LR-max}}=X_n^{\text{LR-max}}(\sigma)$ denote the left-to-right maximum statistic, that is, the total number of left-to-right maxima in $\sigma\in S_n$. Under the uniform probability measure $P_n$ on $S_n$, it follows from symmetry that the probability that $i$ is a left-to-right maximum is $\frac1i$; thus
$E_nX_n^{\text{LR-max}}=\sum_{j=1}^n\frac1j\sim\log n$.  It is well-known that under the uniform distribution,  the left-to-right maximum statistic has the same distribution as the cycle statistic that counts the total number of cycles in a permutation \cite{B}. The well-known law of large numbers and central limit theorem for the cycle statistic under the uniform distribution
thus also holds for the left-to-right maximum statistic.

The left-to-right maximum statistic behaves very differently under $P_n^{\text{park}}$.  Note that it is immediate from Theorem \ref{jnlastcars} that
$E_n^{\text{park}}X_n^{\text{LR-max}}\ge \omega_n$, for large $n$,  if $\omega_n=o(n^\frac16)$.
 In fact, we shall see that $E_n^{\text{park}}X_n^{\text{LR-max}}$ grows on an order at least $n^\frac12$ and no more than $n^{\frac12+\epsilon}$, for any $\epsilon>0$.
The theorem below gives an exact formula for the probability that $i$ is a left-to-right maximum and that also $\sigma_i=j$, from which an exact formula for
$E_n^{\text{park}}X_n^{\text{LR-max}}$ follows. (Of course, the probability that $i$ is a left-to-right maximum and that $\sigma_i=j$ is equal to zero if $i>j$.)
\begin{theorem}\label{lrmaxthm}
\noindent i.
\begin{equation}\label{lrmaxform}
\begin{aligned}
&P_n^{\text{park}}\left(i\ \text{is a left-to-right maximum and}\ \sigma_i=j\right)=\\
&\binom{j-1}{i-1}i^{i-1}(n-i+1)^{j-i-1}\left(i(n-j)(n+1)^{-j}+(n+1)^{1-j}\right),\ 1\le i\le j\le n.
\end{aligned}
\end{equation}
\noindent ii.
\begin{equation}\label{explrmax}
\begin{aligned}
&E_n^{\text{park}}X_n^{\text{LR-max}}=\\
&\sum_{j=1}^n\sum_{i=1}^j\binom{j-1}{i-1}i^{i-1}(n-i+1)^{j-i-1}\left(i(n-j)(n+1)^{-j}+(n+1)^{1-j}\right).
\end{aligned}
\end{equation}
\end{theorem}
We didn't find the right hand side of \eqref{explrmax} very amenable to direct asymptotic analysis. However, we were able to express \eqref{explrmax} in a different form that is more tractable for such analysis. We have the following theorem.
\begin{theorem}\label{lrmaxasympthm}
\noindent i.
\begin{equation}\label{lrmaxformagain}
E_n^{\text{park}}X_n^{\text{LR-max}}=n-\sum_{l=1}^n\frac{n-l}{l(l+1)}\frac{\prod_{j=0}^{l-1}(n-j)}{(n+1)^l}.
\end{equation}
\noindent ii.
\begin{equation}\label{lrmasympform}
\begin{aligned}
&\limsup_{n\to\infty}\frac{E_n^{\text{park}}X_n^{\text{LR-max}}}{n^{\frac12+\epsilon}}=0,\ \text{for any}\ \epsilon>0;\\ &\liminf_{n\to\infty}\frac{E_n^{\text{park}}X_n^{\text{LR-max}}}{n^\frac12}>0.
\end{aligned}
\end{equation}
\end{theorem}

\medskip

The proof of Theorem \ref{ptwisedist} is given in section \ref{secpfptwise} and the proof of Theorem \ref{llnthm} is given in section \ref{secllnthm}. The proof of Theorem \ref{lastcars} is given in
section \ref{lastcarsproof}, the proof of Corollary \ref{genj} is given in section \ref{genjproof} and the proof of Theorem \ref{jnlastcars} is given in section \ref{jnlastcarsproof}.
The proof of Theorem \ref{lrmaxthm} is given in section \ref{lrmaxthmproof} and the proof of Theorem \ref{lrmaxasympthm} is given in sections \ref{lrmaxasympthmproofi} and
\ref{lrmaxasympthmproofii}.

\section{Proof of Theorem \ref{ptwisedist}}\label{secpfptwise}
Recall the definitions of $l_{n,i}$ and $\tilde l_{n,i}$ in \eqref{littlel} and \eqref{littleltilde}.
For the proof of the theorem, it will be convenient to work with
$\tilde l_{n,i}$.
Recall from \eqref{Ln} that
$$
L_n(\sigma)=\prod_{i=1}^n\tilde l_{n,i}(\sigma).
$$

The theorem will follow if we show that for each $\sigma\in S_n$, there are $L_n(\sigma)$ different parking functions $\pi\in \mathcal{P}_n$ such that $T_n(\pi)=\sigma$, where
$T_n$ is as in the paragraph containing equation \eqref{defprob}.
Before giving a formal proof of the theorem, we illustrate the proof with a concrete example, from which the general result should be clear.
Consider the permutation $\sigma=379218645\in S_9$.
We look for those $\pi\in \mathcal{P}_9$ that satisfy $T_9(\pi)=\sigma$.
From the definition of the parking process and from the definition of $T_n$, we need $\pi_1=5$ in order to have $\sigma_5=1$,  $\pi_2=4$ in order to have $\sigma_4=2$, $\pi_3=1$ in order to have $\sigma_1=3$ and $\pi_4=8$ in order to have $\sigma_8=4$.  In order to have $\sigma_9=5$, we can either have $\pi_5=9$, in which case car number 5 parks in its preferred space 9, or alternatively,
$\pi_5=8$, in which case car number 5 attempts to park in its preferred space 8 but fails, and then moves on to space 9 and parks. Then we need $\pi_6=7$ in order to have $\sigma_7=6$.
Then similar to the explanation regarding $\pi_5$, we need $\pi_7$ to be either 1 or 2 in order to have $\sigma_2=7$. In order to have
$\sigma_6=8$, we can have either $\pi_8=6$, in which car number 8 parks directly in its preferred space 6, or alternatively $\pi_8=5$, in which case car number 8 tries and fails to park in space number 5 and then parks in space number 6, or alternatively, $\pi_8=4$, it which case car number 8 tries and fails to park in space number 4 and then also in space number 5, before finally parking in space number 6. Similarly, we need $\pi_9$ to be equal to 1,2 or 3 in order to have $\sigma_9=3$.
Thus, there are $1\times1\times1\times1\times 2\times1\times 2\times 3\times 3=\prod_{i=1}^9\tilde l_{9,i}(\sigma)=L_9(\sigma)$ different parking functions $\pi\in \mathcal{P}_9$ that yield $T_9(\pi)=\sigma$.

To give a formal proof for the general case, fix $\sigma\in S_n$. In order to have $T_n(\pi)=\sigma$, first we need $\pi_1=\sigma^{-1}_1$. Thus there is just one choice for
$\pi_1$, and note that $\tilde l_{n,1}(\sigma)=1$. Now let $k\in[n-1]$ and assume that we have chosen $\pi_1,\cdots\pi_k$ in such a way that car number $i$ has parked in space
$\sigma^{-1}_i$, for $i\in[k]$. We now want  car number $k+1$ to park in space $\sigma^{-1}_{k+1}$. By construction, this space is vacant at this point, and so are the $\tilde l_{n,k+1}(\sigma)-1$ spaces immediately to the left of this space. However the space $\tilde l_{n,k+1}$ spaces to the left of this space is not vacant (or possibly this space doesn't exist--it would be the zeroth space). Thus, by the parking process,  car number $k+1$ will park in space $\sigma^{-1}_{k+1}$ if and only if  $\pi_{k+1}$ is equal to one of the $\tilde l_{n,k+1}(\sigma)$ numbers $\sigma^{-1}_{k+1},\sigma^{-1}_{k+1}-1,\cdots, \sigma^{-1}_{k+1}-\tilde l_{n,k+1}(\sigma)+1$.
This shows that there are $L_n(\sigma)=\prod_{i=1}^n\tilde l_{n,i}(\sigma)$ different parking functions $\pi$ satisfying $T_n(\pi)=\sigma$.
\hfill $\square$

\section{Proof of Theorem  \ref{llnthm}}\label{secllnthm}
We begin with several preliminary results. Recall that $P_n$ is the uniform probability measure on $S_n$.
\begin{lemma}\label{lemmadistlni}
\begin{equation}\label{distlni}
P_n(l_{n,i}=j)=\begin{cases} \frac1j-\frac1{j+1}=\frac1{j(j+1)},\ j=1,\cdots, i-1;\\ \frac1i, \ j=i.\end{cases}.
\end{equation}
\end{lemma}
\begin{proof}
Fix $i$ and let $j\in[i]$. The event  $\{l_{n,i}(\sigma)\ge j\}$ is the event $\big\{\sigma_i=\max\{\sigma_i,\sigma_{i-1}, \cdots, \sigma_{i-j+1}\}\big\}$.
Since $P_n$ is the uniform distribution on $S_n$, we have
\begin{equation}\label{uppertail}
P_n(l_{n,i}\ge j)=\frac1j,\ i\in[n],\ 1\le j\le i.
\end{equation}
  The lemma now follows.
\end{proof}

We now write
\begin{equation}\label{logLn}
\mathcal{S}_n:=\log L_n=\sum_{i=1}^n\log l_{n,i}.
\end{equation}
From Lemma \ref{lemmadistlni}, we have
\begin{equation}\label{explni}
E_n\log l_{n,i}=\sum_{j=1}^{i-1}\frac{\log j}{j(j+1)}+\frac{\log i}i.
\end{equation}
Note that $E_n\log l_{n,i}$ does not depend on $n$, but of course it is only defined for $1\le i\le n$.
\begin{lemma}\label{lemmaexpvaluelim}
\begin{equation}\label{expvaluelim}
\lim_{n,i\to\infty}E_n\log l_{n,i}=\log 2.
\end{equation}
\end{lemma}
\begin{proof}
Recall the Abel-type summation formula \cite{P14}:
$$
\sum_{1<r\le x}a(r)f(r)=A(x)f(x)-A(1)f(1)-\int_1^x A(t)f'(t)dt, \ \text{where}\ A(r)=\sum_{i=1}^ra_i.
$$
We apply  this formula with $a(r)=\frac1{r(r+1)}=\frac1r-\frac1{r+1}$ and  $f(r)=\log r$. We have
$A(r)=1-\frac1{r+1}=\frac r{r+1}$.
 Recalling \eqref{explni}, we obtain
$$
\begin{aligned}
&\lim_{n,i\to\infty}E_n\log l_{n,i}=\lim_{i\to\infty}\sum_{j=1}^{i-1}\frac{\log j}{j(j+1)}=\lim_{i\to\infty}\left(\frac i{i+1}\log i-\int_1^i\frac t{t+1}\frac1tdt\right)=\\
&\lim_{i\to\infty}\left(\frac i{i+1}\log i-\log(i+1)+\log2\right)=\lim_{i\to\infty}\left(\log\frac i{i+1}-\frac{\log i}{i+1}+\log2\right)=\log 2.
\end{aligned}
$$
\end{proof}
From \eqref{logLn} and \eqref{expvaluelim}, we conclude that
\begin{equation}\label{limScaln}
\lim_{n\to\infty}\frac{E_n\mathcal{S}_n}n=\lim_{n\to\infty}\frac1nE_n\log L_n=\log 2.
\end{equation}

We now consider $E_n\mathcal{S}_n^2$. We have
\begin{equation}\label{2ndmom}
E_n\mathcal{S}_n^2=E_n\left(\sum_{i=1}^n\log l_{n,i}\right)^2=\sum_{i=1}^nE_n\log l_{n,i}^2\thinspace+2\sum_{1\le i<j\le n}E_n\log l_{n,i}\log l_{n,j}.
\end{equation}
We have the following proposition.
\begin{proposition}
For $1\le i<j\le n$, the random variables  $l_{n,i}$ and $l_{n,j}$ on $(S_n,P_n)$ are negatively correlated; that is,
\begin{equation}\label{negcorr}
P_n(l_{n,i}\ge k,\l_{n,j}\ge k)\le P_n(l_{n,i}\ge k)P_n(l_{n,j}\ge l), \ \text{for}\ k,l\ge1.
\end{equation}
\end{proposition}
\begin{proof}
Since $P_n$ is the uniform probability measure on $S_n$,  for any $k\le i$,   the events $\{l_{n,i}\ge k\}=\big\{\sigma_i=\max(\sigma_i,\cdots, \sigma_{i-k+1})\big\}$ and
$\{l_{n,j}\ge l\}=\big\{\sigma_j=\max(\sigma_j,\cdots, \sigma_{j-l+1})\big\}$ are independent if $l\le j-i$.
Thus, \eqref{negcorr} holds with equality in these cases.

Consider now the case $k\le i$ and $j-i+1\le l\le j$. In this case
\begin{equation}\label{intersectionevents}
\begin{aligned}
&\{l_{n,i}\ge k,l_{n,j}\ge l\}=\big\{\sigma_j=\max(\sigma_j,\sigma_{j-1},\cdots, \sigma_r)\big\}\cap\big\{\sigma_i=\max(\sigma_i,\sigma_{i-1},\cdots, \sigma_{i-k+1})\big\},\\
& \text{where}\
r=\min(i-k+1,j-l+1).
\end{aligned}
\end{equation}
We have
\begin{equation}\label{firstprob}
P_n\left(\sigma_j=\max(\sigma_j,\sigma_{j-1},\cdots, \sigma_r)\right)=\frac1{\max(l,j-i+k)}\le \frac1l.
\end{equation}
Also,
\begin{equation}\label{cond2ndprob}
P_n\left(\sigma_i=\max(\sigma_i,\sigma_{i-1},\cdots,\sigma_{i-k+1} |\sigma_j=\max(\sigma_j,\sigma_{j-1},\cdots, \sigma_r)\right)=\frac1k,
\end{equation}
because $\{\sigma_i,\cdots, \sigma_{i-k+1}\}\subset\{\sigma_j,\sigma_{j-1},\cdots, \sigma_r\}$.
The proposition follows from \eqref{intersectionevents}--\eqref{cond2ndprob} and \eqref{uppertail}.
\end{proof}
We can now prove the theorem.

\noindent \it Proof of Theorem \ref{llnthm}.\rm\
Since $l_{n,i}$ and $l_{n,j}$ are negatively correlated, one has $E_nf(l_{n,i})g(l_{n,j})\le E_nf(l_{n,i})E_ng(l_{n,j})$, if $f$ and $g$ are increasing functions on $[n]$.
In particular then,
\begin{equation}\label{covar}
E_n\log l_{n,i}\log l_{n,j}\le E_n\log l_{n,i}E_n\log l_{n,j}.
\end{equation}
From \eqref{logLn} and  \eqref{covar}, a standard straightforward calculation gives
\begin{equation}\label{varsuminequal}
\text{\rm Var}(\mathcal{S}_n)\le\sum_{i=1}^n\text{\rm Var}(\log l_{n,i}).
\end{equation}
From \eqref{distlni}, we have
$$
E_n(\log l_{n,i})^2=\sum_{j=1}^{i-1}\frac{(\log j)^2}{j(j+1)}+\frac{(\log i)^2}{i(i+1)}.
$$
Using this with \eqref{expvaluelim} and \eqref{varsuminequal}, we conclude that there exists a $C>0$ such that
\begin{equation}\label{varbound}
\text{\rm Var}(\mathcal{S}_n)\le Cn,\ n\in\mathbb{N}.
\end{equation}
From \eqref{limScaln} and\eqref{varbound}, it follows from the second moment method (Chebyshev's inequality) that
\begin{equation}\label{WLLN}
\lim_{n\to\infty}P_n(\log2-\epsilon\le \frac{\mathcal{S}_n}n\le\log2+\epsilon)=1,\ \text{for all}\ \epsilon>0.
\end{equation}
Now \eqref{wlln} follows from \eqref{WLLN} and \eqref{logLn}.
\hfill $\square$

\section{Proof of Theorem \ref{lastcars}}\label{lastcarsproof}
We count how many parking functions $\{\pi_l\}_{l=1}^n$  yield  the event $\{\sigma^{-1}_{n-j+1}=i_1,\sigma^{-1}_{n-j+2}=i_2,\cdots, \sigma^{-1}_n=i_j\}$. In order for this event to occur,  the first $n-j$ cars have to park arbitrarily in the spots
$[n]-\{k_l\}_{l=1}^j$. Then the last $j$ cars have to park, with each one in its appropriate space.

We now analyze the parking of the first $n-j$ cars.
In order for $k_1-1$ cars from among   the first  $n-j$  cars to park in the first $k_1-1$ spaces, but for no one of the first $n-j$ cars  to park in the $k_1$th space,
it follows from the definition of the parking process that
there must be a collection of $k_1-1$ of the $\{\pi_i\}_{l=1}^{n-j}$ which constitute a parking function of length $k_1-1$. We can freely choose which of the first $n-j$ cars to use for these spaces.
Since there are $k_1^{k_1-2}$ parking functions of length $k_1-1$, this gives $\binom{n-j}{k_1-1}k_1^{k_1-2}$ different choices.

In order for $k_2-k_1-1$ cars, from among   the  $n-j-(k_1-1)$ cars that still remain from among the first $n-j$ cars, to park in the spaces $k_1+1,\cdots, k_2-1$, but for no one of the first $n-j$ cars to park in the $k_1$th space or the $k_2$th space,
there must be a collection  of $k_2-k_1-1$ of the remaining $n-j-(k_1-1)$ members of $\{\pi_l\}_{l=1}^{n-j}$ that constitute a parking function of length $k_2-k_1-1$, but shifted forward by $k_1$ spaces.
(If $\{\pi_l\}_{l=1}^m$ is  a parking function of length $m$, then we call the collection $\{a+\pi_l\}_{l=1}^m$, where $a\in\mathbb{N}$, a parking function of length $m$ shifted forward by $a$ spaces. Obviously, the number of such parking functions coincides with the number of parking functions of length $m$.)
We can freely choose which of the remaining $n-j-(k_1-1)$ cars from among the first $n-j$ cars to  use for these spaces. Thus, there are $\binom{n-j-(k_1-1)}{k_2-k_1-1}(k_2-k_1)^{k_2-k_1-2}$ different choices.

We continue counting in this fashion until we have chosen cars for the $k_j-k_{j-1}-1$ spaces between spaces $k_{j-1}$ and $k_j$. Recall that $k_0$ has been defined to be 0. Now we need for all of the $n-j-\sum_{i=1}^j(k_i-k_{i-1}-1)=n-k_j$ cars that remain from among the first $n-j$ cars to park in the spaces $k_j+1,\cdots, n$.
In order for this to occur, but for no one of the first $n-j$ cars to park in the $k_j$th space, these remaining members of $\{\pi_i\}_{l=1}^{n-j}$  must   constitute a parking function of length $n-k_j$, but shifted forward by $k_j$ spaces. This gives $(n-k_j+1)^{n-k_j-1}$ different choices.

Thus, we have shown that the number of ways to choice $\{\pi_l\}_{l=1}^{n-j}$ is
\begin{equation}\label{n-jcount}
\begin{aligned}
&\binom{n-j}{k_1-1}k_1^{k_1-2}\binom{n-j-k_1+1}{k_2-k_1-1}(k_2-k_1)^{k_2-k_1-2}\cdots\times\\
&\binom{n-j-\sum_{l=1}^{j-1}(k_l-k_{l-1})}{k_j-k_{j-1}-1}(k_j-k_{j-1})^{k_j-k_{j-1}-2}(n-k_j+1)^{n-k_j-1}=\\
&\frac{(n-j)!}{(n-k_j)!\prod_{l=1}^j(k_l-k_{l-1}-1)!}\thinspace(n-k_j+1)^{n-k_j-1}\prod_{l=1}^j(k_l-k_{l-1})^{k_l-k_{l-1}-2}.
\end{aligned}
\end{equation}

Now we count how many ways we can choice the $\{\pi_l\}_{l=n-j+1}^n$ in order that $\sigma_{n-j+l}=i_l,\ l=1,\cdots, j$, or equivalently, in order that car $n-j+l$ park in space $i_l,\ l=1,\cdots, j$. This calculation uses the same reasoning that was used in the proof of Theorem \ref{ptwisedist}.
After the first $n-j$ cars have parked appropriately as above, the spaces
$\{k_l\}_{l=1}^j=\{i_l\}_{l=1}^j$ are still vacant. From the definition of the parking process and from the definition of $\tilde l_{n,n-j+1}(\sigma)$,
in order that car $n-j+1$ park in space $i_1$,  $\pi_{n-j+1}$ must take on one of the values $i_1.i_1-1,\cdots i_1-\tilde l_{n,n-j+1}(\sigma)+1$.
Then in order that car $n-j+2$ park in space $i_2$, $\pi_{n-j+2}$ must take on one of the values $i_2,i_2-1,\cdots i_2-\tilde l_{n,n-j+2}(\sigma)+1$.
Continuing like this, we conclude that there are $\prod_{l=1}^j\tilde l_{n,n-j+l}(\sigma)$ choices.
Using this with \eqref{n-jcount}, and noting that there are $(n+1)^{n-1}$ different parking functions of length $n$, we obtain \eqref{lastcarsformula}.
\hfill $\square$

\section{Proof of Corollary \ref{genj}}\label{genjproof}
Let $\{m_r\}_{r=1}^j\subset\mathbb{N}$.
By Theorem \ref{lastcars}, for $n\ge\sum_{r=1}^jm_r$,
the probability\newline
 $P_n^{\text{park}}(n+1-\sigma^{-1}_{n-j+l}=\sum_{r=l}^jm_r;\ l=1,\cdots, j)$
is given by
the formula on the right hand side of \eqref{lastcarsformula} with
 with $i_l=n+1-\sum_{r=l}^jm_r,\ l\in[j]$. Since these $\{i_l\}_{l=1}^j$ are increasing, we have $k_l=i_l, \ l\in[j]$.
So we have
\begin{equation}\label{kdifferences}
\begin{aligned}
&k_l-k_{l-1}=m_{l-1},\ l=2,\cdots j; \\
 &k_1-k_0=n+1-\sum_{r=1}^jm_r.
\end{aligned}
\end{equation}
Since the $\{i_l\}_{l=1}^j$ are increasing and represent the positions respectively of the numbers  $\{n-j+l\}_{l=1}^j$ in the permutation $\sigma$, it follows
from the definition of $\tilde l_{n,\cdot}(\sigma)$ in \eqref{littleltilde} that
\begin{equation}\label{ltildevalues}
\tilde l_{n,n-j+l}(\sigma)=i_l=n+1-\sum_{r=l}^jm_r.
\end{equation}
Substituting from \eqref{kdifferences} and \eqref{ltildevalues} into the right hand side of  \eqref{lastcarsformula}, the expression there becomes
\begin{equation}\label{bigexpression}
\begin{aligned}
&\frac{(n-j)!\left(n+1-\sum_{r=1}^jm_r\right)^{n-1-\sum_{r=1}^jm_r}\left(\prod_{l=2}^jm_{l-1}^{m_{l-1}-2}\right)m_j^{m_j-2}\prod_{l=1}^j\left(n+1-\sum_{r=l}^jm_r\right)}
{(n+1)^{n-1}(n-\sum_{r=1}^jm_r)!\left(\prod_{l=2}^j(m_{l-1}-1)!\right)(m_j-1)!}.
\end{aligned}
\end{equation}
We rewrite the terms in \eqref{bigexpression} that do not involve $n$ as
\begin{equation}\label{notn}
\frac{\left(\prod_{l=2}^jm_{l-1}^{m_{l-1}-2}\right)m_j^{m_j-2}}{\left(\prod_{l=2}^j(m_{l-1}-1)!\right)(m_j-1)!}=\prod_{l=1}^j\frac{m_l^{m_l-1}}{m_l!}.
\end{equation}
The terms in \eqref{bigexpression} that involve $n$ satisfy
\begin{equation}\label{withn}
\begin{aligned}
&\lim_{n\to\infty}\frac{(n-j)!\left(n+1-\sum_{r=1}^jm_r\right)^{n-1-\sum_{r=1}^jm_r}\prod_{l=1}^j\left(n+1-\sum_{r=l}^jm_r\right)}{(n+1)^{n-1}(n-\sum_{r=1}^jm_r)!}=\\
&e^{-\sum_{r=1}^jm_r}.
\end{aligned}
\end{equation}
From \eqref{bigexpression}-\eqref{withn}, it follows that the right hand side of \eqref{lastcarsformula} with $i_l=n+1-\sum_{r=l}^jm_r,\ l\in[j]$, converges as $n\to\infty$ to
$\prod_{l=1}^j\frac{m_l^{m_l-1}e^{-m_l}}{m_l!}$.
Thus, we have proven that
\begin{equation}\label{weaklimitproduct}
\lim_{n\to\infty}P_n^{\text{park}}(n+1-\sigma^{-1}_{n-j+l}=\sum_{r=l}^jm_r;\ l=1,\cdots, j)=\prod_{l=1}^j\frac{m_l^{m_l-1}e^{-m_l}}{m_l!}.
\end{equation}
The corollary follows from this.
\hfill $\square$

\section{Proof of Theorem \ref{jnlastcars}}\label{jnlastcarsproof}
From the proof of Corollary \ref{genj} up to \eqref{notn}, in particular, from \eqref{bigexpression} and \eqref{notn}, we have for $j_n\le n$,
\begin{equation}\label{bigexpressionjn}
\begin{aligned}
&P_n^{\text{park}}\left(n+1-\sigma^{-1}_{n-j_n+1}=\sum_{r=1}^{j_n}m_r,n+1-\sigma^{-1}_{n-j_n+2}=\sum_{r=2}^{j_n}m_r,\cdots,n+1-\sigma^{-1}_n=m_{j_n}\right)=\\
&\left(\prod_{r=1}^{j_n}\frac{m_r^{m_r-1}e^{-m_r}}{m_r!}\right)A(n,j_n;m_1\cdots,m_{j_n})B(n,j_n;m_1,\cdots,m_{j_n}),\ \text{for}\ \sum_{r=1}^{j_n}m_r\le n,\ m_r\ge1,
\end{aligned}
\end{equation}
where
\begin{equation}\label{AB}
\begin{aligned}
&A(n,j;m_1\cdots,m_j)=e^{\sum_{r=1}^jm_r}\left(\frac{n+1-\sum_{r=1}^jm_r}{n+1}\right)^{n-1}=
e^{\sum_{r=1}^jm_r}\left(1-\frac{\sum_{r=1}^jm_r}{n+1}\right)^{n-1};\\
& B(n,j;m_1,\cdots, m_j)=\frac{(n-j)!}{(n-\sum_{r=1}^jm_r)!}\frac{\prod_{r=1}^j(n+1-\sum_{r=l}^jm_r)}{(n+1-\sum_{r=1}^jm_r)^{\sum_{r=1}^jm_r}}.
\end{aligned}
\end{equation}

By Stirling's formula, $\frac{m^{m-1}e^{-m}}{m!}\sim\frac1{\sqrt{2\pi}}m^{-\frac32}$.
Thus, there exists a $C>0$ such that
$$
\sum_{m=N+1}^\infty\frac{m^{m-1}e^{-m}}{m!}<\frac C{N^{\frac12}},\ N\in\mathbb{N}.
$$
Since $\{\frac{m^{m-1}e^{-m}}{m!}\}_{m=1}^\infty$ is a probability distribution, it then follows that
\begin{equation}\label{wotail}
\sum_{m=1}^N\frac{m^{m-1}e^{-m}}{m!}\ge1-\frac C{N^\frac12},\ N\in\mathbb{N}.
\end{equation}
Using  \eqref{wotail} for the inequality below, we have
\begin{equation}\label{wotailproduct}
\sum_{\stackrel{1\le m_r\le N}{r=1,\cdots, j_n}}\prod_{r=1}^N\frac{m_r^{m_r-1}e^{-m_r}}{m_r!}=\left(\sum_{m=1}^N\frac{m^{m-1}e^{-m}}{m!}\right)^{j_n}\ge\left(1-\frac C{N^\frac12}\right)^{j_n}.
\end{equation}
 Substituting $N_n$ for $N$ in \eqref{wotailproduct},  a standard estimate gives
\begin{equation}\label{probgoesto0}
\lim_{n\to\infty}\sum_{\stackrel{1\le m_r\le N_n}{r=1,\cdots, j_n}}\prod_{r=1}^{N_n}\frac{m_r^{m_r-1}e^{-m_r}}{m_r!} =1,\ \text{if}\ j_n=o(N_n^\frac12).
\end{equation}

We can express $P_n^{\text{park}}(\sigma^{-1}_{n+1-j_n}<\sigma^{-1}_{n+1-j_n+1}<\cdots<\sigma^{-1}_n)$ as
$$
\begin{aligned}
&P_n^{\text{park}}(\sigma^{-1}_{n+1-j_n}<\sigma^{-1}_{n+1-j_n+1}<\cdots<\sigma^{-1}_n)=\\
&\sum_{\stackrel{m_r\ge1;r=1,\cdots, j_n}{\sum_{r=1}^{j_n}m_r\le n}}P_n^{\text{park}}\left(n+1-\sigma^{-1}_{n-j_n+l}=\sum_{r=l}^{j_n}m_r,\ l=1,\cdots j_n\right).
\end{aligned}
$$
Using this with \eqref{bigexpressionjn}, we have
\begin{equation}\label{giant}
\begin{aligned}
&P_n^{\text{park}}(\sigma^{-1}_{n+1-j_n}<\sigma^{-1}_{n+1-j_n+1}<\cdots<\sigma^{-1}_n)=\\
&\sum_{\stackrel{m_r\ge1;r=1,\cdots, j_n}{\sum_{r=1}^{j_n}m_r\le n}}\left(\prod_{r=1}^{j_n}\frac{m_r^{m_r-1}e^{-m_r}}{m_r!}\right)A(n,j_n;m_1\cdots,m_{j_n})B(n,j_n;m_1,\cdots,m_{j_n}).
\end{aligned}
\end{equation}
If $m_r\le N_n$ for $r=1,\cdots, j_n$, then
 trivially, $\sum_{r=1}^{j_n}m_r\le j_nN_n$. Thus,
if we choose $j_n$ and $N_n$ such that
\begin{equation}\label{ABto1}
\begin{aligned}
&\lim_{n\to\infty}A(n,j_n;m_1\cdots,m_{j_n})=1,\ \text{uniformly over}\ \{1\le m_r\le N_n;r=1,\cdots, j_n\};\\
&\lim_{n\to\infty}B(n,j_n;m_1,\cdots,m_{j_n})=1, \ \text{uniformly over}\ \{1\le m_r\le N_n;r=1,\cdots, j_n\},
\end{aligned}
\end{equation}
and such that
\begin{equation}\label{jnNnrelation}
j_nN_n\le n;\ \text{and}\   j_n=o(N_n^\frac12),
\end{equation}
then it will follow from \eqref{probgoesto0}-\eqref{ABto1} that
\begin{equation}\label{goal}
\lim_{n\to\infty}P_n^{\text{park}}(\sigma^{-1}_{n+1-j_n}<\sigma^{-1}_{n+1-j_n+1}<\cdots<\sigma^{-1}_n)=1.
\end{equation}

In the calculations below, we always assume that $1\le m_r\le N_n$, for $r=1,\cdots, j_n$.
We first consider $A(n,j_n;m_1\cdots,m_{j_n})$ in \eqref{AB}.
From Taylor's remainder formula, we have
\begin{equation}\label{Taylor}
-x\ge\log(1-x)\ge-x-\frac12x^2\frac1{(1-x)^2}, \ x\in(0,1).
\end{equation}
Using \eqref{Taylor}, we have
\begin{equation}\label{partofA}
\begin{aligned}
&-\frac{n-1}{n+1}\sum_{r=1}^{j_n}m_r\ge\log\left(1-\frac{\sum_{r=1}^{j_n}m_r}{n+1}\right)^{n-1}\ge\\
&-\frac{n-1}{n+1}\sum_{r=1}^{j_n}m_r-
(n-1)\left(\frac12\frac{(\sum_{r=1}^{j_n}m_r)^2}{(n+1-\sum_{r=1}^{j_n}m_r)^2}\right).
\end{aligned}
\end{equation}
From \eqref{AB} and \eqref{partofA}, we have
\begin{equation}\label{Aest}
\begin{aligned}
&\frac2{n+1}\sum_{r=1}^{j_n}m_r\ge\log A(n,j_n;m_1\cdots,m_{j_n})\ge\\
&\frac2{n+1}\sum_{r=1}^{j_n}m_r-\frac{(n-1)(\sum_{r=1}^{j_n}m_r)^2}{2(n+1-\sum_{r=1}^{j_n}m_r)^2}.
\end{aligned}
\end{equation}
Recalling that $\sum_{r=1}^{j_n}m_r\le j_nN_n$, it follows from \eqref{Aest} that
\begin{equation}\label{Ato1}
\eqref{ABto1}\ \text{holds for}\ A(n,j_n;m_1\cdots,m_{j_n}) \ \text{if}\ j_nN_n=o(n^\frac12).
\end{equation}

We now turn to $B(n,j_n;m_1\cdots,m_{j_n})$ in \eqref{AB}.
We write
\begin{equation}\label{Bform}
B(n,j_n;m_1\cdots,m_{j_n})=C(n,j_n;m_1\cdots,m_{j_n})D(n,j_n;m_1\cdots,m_{j_n}),
\end{equation}
where
\begin{equation}\label{CD}
\begin{aligned}
&C(n,j_n;m_1\cdots,m_{j_n})=\frac{(n-j_n)!}{(n-\sum_{r=1}^{j_n}m_r)!}n^{j_n-\sum_{r=1}^{j_n}m_r};\\
&D(n,j_n;m_1\cdots,m_{j_n})=\frac{\prod_{r=1}^{j_n}(n+1-\sum_{r=l}^{j_n}m_r)}{(n+1-\sum_{r=1}^{j_n}m_r)^{\sum_{r=1}^{j_n}m_r}}n^{\sum_{r=1}^{j_n}m_r-j_n}.
\end{aligned}
\end{equation}
Let $w_n=\sum_{r=1}^{j_n}m_r$. We suppress the dependence of $w_n$ on $\{m_r\}_{r=1}^{j_n}$. For any choice of $\{m_r\}_{r=1}^{j_n}$,   we have
$j_n\le w_n\le j_nN_n$.

We have
$$
(n-w_n)^{w_n-j_n}n^{j_n-w_n}\le C(n,j_n;m_1\cdots,m_{j_n})\le(n-j_n)^{w_n-j_n}n^{j_n-w_n},
$$
or equivalently,
\begin{equation}\label{Cest}
\left(1-\frac{w_n}n\right)^{w_n-j_n}\le C(n,j_n;m_1\cdots,m_{j_n})\le\left(1-\frac{j_n}n\right)^{w_n-j_n}.
\end{equation}
Since $w_n\le j_nN_n$, it follows from \eqref{Cest} that
\begin{equation}\label{Cto1}
\begin{aligned}
&\text{if}\ j_nN_n=o(n^\frac12), \ \text{then}\\
&\lim_{n\to\infty}C(n,j_n;m_1\cdots,m_{j_n})=1, \ \text{uniformly over}\ \{1\le m_r\le N_n;r=1,\cdots, j_n\}.
\end{aligned}
\end{equation}

We have
$$
\frac{(n+1-w_n)^{j_n}}{(n+1-w_n)^{w_n}}n^{w_n-j_n}\le D(n,j_n;m_1\cdots,m_{j_n})\le \frac{n^{j_n}}{(n+1-w_n)^{w_n}}n^{w_n-j_n},
$$
or equivalently,
\begin{equation}\label{Dest}
\left(1-\frac{w_n-1}n\right)^{j_n}\left(1+\frac{w_n-1}{n+1-w_n}\right)^{w_n}\le D_n\le\left(1+\frac{w_n-1}{n+1-w_n}\right)^{w_n}.
\end{equation}
Since $w_n\le j_nN_n$, it follows from \eqref{Dest} that
\begin{equation}\label{Dto1}
\begin{aligned}
&\text{if}\ j_nN_n=o(n^\frac12), \ \text{then}\\
&\lim_{n\to\infty}D(n,j_n;m_1\cdots,m_{j_n})=1, \ \text{uniformly over}\ \{1\le m_r\le N_n;r=1,\cdots, j_n\}.
\end{aligned}
\end{equation}
From  \eqref{Bform}, \eqref{Cto1} and \eqref{Dto1}, it follows that
\begin{equation}\label{Bto1}
\eqref{ABto1}\ \text{holds for}\ B(n,j_n;m_1\cdots,m_{j_n}) \ \text{if}\ j_nN_n=o(n^\frac12).
\end{equation}

From \eqref{Ato1}, \eqref{Bto1} and \eqref{jnNnrelation}, we conclude that \eqref{goal} holds if
$j_nN_n=o(n^\frac12)$ and $j_n=o(N_n^\frac12)$.
Thus, if $j_n=o(n^\frac16)$ and $N_n=[n^\frac13]$, then \eqref{goal} holds. This completes the proof of the theorem.
\hfill$\square$

\section{Proof of Theorem \ref{lrmaxthm}}\label{lrmaxthmproof}
Part (ii) of the theorem follows immediately from part (i). For part (i),
let $i,j$ satisfy $1\le i\le j\le n$.
 We need to  calculate the   $P_n^{\text{\rm park}}$-probability that  $\sigma_i=j$ and that the location $i$ is a left-to-right maximum for $\sigma$.
To do this, we count how many parking functions $\{\pi_k\}_{k=1}^n$ yield a permutation $\sigma\in S_n$ with the above properties.
A permutation $\sigma\in S_n$ satisfies the above properties if and only if  $\{\sigma_1,\cdots, \sigma_{i-1}\}\subset [j-1]$ and $\sigma_i=j$.
It follows from the definition of the parking process that a parking function
$\{\pi_k\}_{k=1}^n$  yields such a permutation if and only if a subsequence of
  $\{\pi_1,\cdots, \pi_{j-1}\}$ of length $i-1$   constitutes
a parking function of size $i-1$, the other $j-i$ elements of $\{\pi_1,\cdots, \pi_{j-1}\}$ belong to  $[n]-[i]$,  and
 $\pi_j\in[i]$.
So we need to count how many such parking functions there are.

There are $\binom{j-1}{i-1}$ choices of $i-1$ elements for the subsequence of length $i-1$ from
$\{\pi_1,\cdots, \pi_{j-1}\}$. There are $i^{i-2}$ choices for such a subsequence of length $i-1$ to constitute a parking function of length $i-1$.
There are $i$ choices for $\pi_j$.
So this gives $\binom{j-1}{i-1}i^{i-1}$ choices for
$\pi_j$ and for the $i-1$ elements of $\{\pi_1,\cdots, \pi_{j-1}\}$
that are restricted to $[i]$.

 Denote by
  $\{\pi'_1,\cdots, \pi'_{j-i}\}$.
  the complementary  subsequence of  $j-i$ elements of
$\{\pi_1,\cdots, \pi_{j-1}\}$.
that all lie in $[n]-[i]$.
The components $\{\pi_{j+1},\cdots,\pi_n\}$ are without specific restrictions.
(But of course,  $\{\pi'_1,\cdots, \pi'_{j-i}\}$ and $\{\pi_{j+1},\cdots,\pi_n\}$ must be chosen so that $\{\pi_k\}_{k=1}^n$ is a parking function.)

In order to count how many choices there are for  $\{\pi'_1,\cdots, \pi'_{j-i}\}$ and $\{\pi_{j+1},\cdots,\pi_n\}$, we fix $l\in[n-j]$ and determine how many choices there are
for $\{\pi'_1,\cdots, \pi'_{j-1}\}$ and $\{\pi_{j+1},\cdots,\pi_n\}$, subject to the condition that a particular subsequence of
 $\{\pi_{j+1},\cdots,\pi_n\}$ of length $l$ has all of its elements in $[i]$, while the remaining elements in
  $\{\pi_{j+1},\cdots,\pi_n\}$ are in $[n]-[i]$.
 Then we will sum over all such $l$.
There are $\binom{n-j}l$  ways to choose
the subsequence of length $l$  from $\{\pi_{j+1},\cdots, \pi_n\}$.
 There are $i^l$ different possible values for these elements.
 Denote the subsequence of length $n-j-l$ of  elements of $\{\pi_{j+1},\cdots,\pi_n\}$ that are not in  $[i]$ by $\{\pi'_{j+m}\}_{m=1}^{n-j-l}$.

We now count how many
choices there are for
$(\pi'_1,\cdots\pi'_{j-i},\pi'_{j+1},\cdots \pi'_{n-l})$.
Our strategy will be to reduce the situation at hand to the following one (with appropriate choice of $n$ and $m$):
There is a one-way street with $n$ spaces, but with the first $m$ of them already taken up by a trailer. A sequence of $n-m$ cars enters, each with a preferred parking space between 1 and $n$. It was shown in \cite{EH} that the number of such sequences resulting in all $n-m$ cars successfully parking   is $(m+1)(n+1)^{n-m-1}$.

By definition, all of the entries of the  sequence  $(\pi'_1,\cdots\pi'_{j-i},\pi'_{j+1},\cdots \pi'_{n-l})$ lie in $[n]-[i]$.
Let $(\bar\pi'_1,\cdots, \bar\pi'_{n-i})$  denote the elements   $(\pi'_1,\cdots, \pi'_{j-i},\pi'_{j+1},\cdots, \pi'_{n-l})$ along with the $l$ fixed elements from $\{\pi_{j+1},\cdots, \pi_n\}$
that are in $[i]$, using the order in which they appear in the  parking function $\{\pi_k\}_{k=1}^n$.
Counting the number of choices for $(\pi'_1,\cdots\pi'_{j-i},\pi'_{j+1},\cdots \pi'_{n-l})$ is equivalent to counting the number of choices for
$(\bar\pi'_1,\cdots, \bar\pi'_{n-i})$, since the latter was obtained from the former by adding fixed components. Now let
$(\pi''_1,\cdots, \pi''_{n-i})$ denote what one obtains if one starts with the sequence  $(\bar\pi'_1,\cdots, \bar\pi'_{n-i})$  and slides the  $l$ entries that are in $[i]$ all the way to the left end of the sequence.
So the first $l$ entries in   $(\pi''_1,\cdots, \pi''_{n-i})$ belong to $[i]$ and the $n-i-l$ other entries are \newline
$(\pi'_1,\cdots, \pi'_{j-i},\pi'_{j+1},\cdots, \pi'_{n-l})$, in the this order.
By the symmetry of parking functions, counting the number of choices for $(\pi'_1,\cdots\pi'_{j-i},\pi'_{j+1},\cdots \pi'_{n-l})$ is equivalent to counting the number of choices for
$(\pi''_1,\cdots, \pi''_{n-i})$.

We now count the number of choices for $(\pi''_1,\cdots, \pi''_{n-i})$. Remember that the first $i$ parking spaces have already been filled. We need to consider how many choices there are  for $(\pi''_1,\cdots, \pi''_{n-i})$ so that the remaining $n-i$ cars can park. By construction, the first $l$ entries in
$(\pi''_1,\cdots, \pi''_{n-i})$ involve no choice; they are fixed. Furthermore, all of these entries are in $[i]$. Thus,
when
using $(\pi''_1,\cdots, \pi''_{n-i})$, the first $l$ cars from the $n-i$ cars    will park in spaces $i+1,\cdots, i+l$.
Now consider the state of affairs at this stage. The first $i+l$ positions are filled. What remains of the parking function is
$(\pi''_{l+1},\cdots, \pi''_{n-i})$, which has length $n-l-i$. All of the entries of
$(\pi''_{l+1},\cdots, \pi''_{n-i})$
 lie in $[n]-[i]$.
This is  equivalent to the
situation described above from  \cite{EH}.
Our case corresponds to the situation in \cite{EH}  with $m$ and $n$ replaced respectively by  $l$ and $n-l-i$.
More specifically, the number of choices for $(\pi''_{l+1},\cdots, \pi''_{n-i})$ is the number of choices for a parking function $(\Pi_1,\cdots, \Pi_{n-l-i})$, all of whose entries have values in $[n-i]$ in the scenario where there are  $n-i$ spaces and $n-i-l$ cars, and  a trailer  takes up the first $l$ spaces.
Thus, from the formula from \cite{EH} noted above, the number of such parking functions is $(l+1)(n-i+1)^{n-i-l-1}$.
So we conclude that there are $(l+1)(n-i+1)^{n-i-l-1}$ choices for $(\pi'_1,\cdots\pi'_{j-i},\pi'_{j+1},\cdots \pi'_{n-l})$.

From the above analysis, we conclude that the
number of parking functions
$\{\pi_k\}_{k=1}^n$
 that yield a permutation $\sigma\in S_n$ for which $\sigma_i=j$ and the location $i$ is a left-to-right maximum is given by
\begin{equation}\label{sigmai=j}
\sum_{l=0}^{n-j}\binom{j-1}{i-1}i^{i-1}\binom{n-j}li^l(l+1)(n-i+1)^{n-i-l-1}.
\end{equation}
A standard calculation gives
\begin{equation}\label{auxi}
\sum_{l=0}^{n-j}(l+1)\binom{n-j}li^l(n-i+1)^{n-j-l}=i(n-j)(n+1)^{n-j-1}+(n+1)^{n-j}.
\end{equation}
Thus, \eqref{sigmai=j} is equal to
$$
\binom{j-1}{i-1}i^{i-1}(n-i+1)^{j-i-1}\left(i(n-j)(n+1)^{n-j-1}+(n+1)^{n-j}\right).
$$
Since there are $(n+1)^{n-1}$ parking functions  $\{\pi_k\}_{k=1}^n$ of length $n$, it follows that
$$
\begin{aligned}
&P_n^{\text{park}}\left(i\ \text{is a left-to-right maximum and}\ \sigma_i=j\right)=\\
&\binom{j-1}{i-1}i^{i-1}(n-i+1)^{j-i-1}\left(i(n-j)(n+1)^{-j}+(n+1)^{1-j}\right),
\end{aligned}
$$
which is
\eqref{lrmaxform}.
\hfill $\square$

\section{Proof of part (i) of Theorem \ref{lrmaxasympthm}}\label{lrmaxasympthmproofi}
We first proof \eqref{lrmaxformagain}.
In light  of Theorem \ref{lrmaxthm}, we need to show that the right hand side of \eqref{explrmax} is equal to the right hand side of \eqref{lrmaxformagain}.
We write the right hand side of \eqref{explrmax} as
\begin{equation}\label{IandII}
\begin{aligned}
&\sum_{j=1}^n\sum_{i=1}^j\binom{j-1}{i-1}i^{i-1}(n-i+1)^{j-i-1}\left(i(n-j)(n+1)^{-j}+(n+1)^{1-j}\right)=I+II,\\
& \text{where}\ I=\sum_{j=1}^n\sum_{i=1}^j\binom{j-1}{i-1}i^i(n-i+1)^{j-i-1}(n-j)(n+1)^{-j}\\
&\text{and}\ II=\sum_{j=1}^n\sum_{i=1}^j\binom{j-1}{i-1}i^{i-1}(n-i+1)^{j-i-1}(n+1)^{1-j}.
\end{aligned}
\end{equation}

In \cite{R}, chapter 1 considers sums of the form
$$
A_n(x,y;p,q):=\sum_{k=0}^n\binom nk(x+k)^{k+p}(y+n-k)^{n-k+q}.
$$
When $p=-1$ and $q=0$, one has
$A_n(x,y;-1,0)=\frac{(x+y+n)^n}x$,
which is known as Abel's generalization of the binomial theorem.
It is easy to see that the symmetry formula $A_n(x,y;p,q)=A_n(y,x;q,p)$ holds; thus,
\begin{equation}\label{Abelbinom}
A_n(x,y;0,-1)=\frac{(x+y+n)^n}y.
\end{equation}
This allows us to calculate $II$.
We have
$$
\begin{aligned}
&\sum_{i=1}^j\binom{j-1}{i-1}i^{i-1}(n-i+1)^{j-i-1}=\sum_{r=0}^{j-1}\binom{j-1}r(r+1)^r(n-r)^{(j-1)-r-1}=\\
&A_{j-1}(1,n-j+1;0,-1)=\frac{(n+1)^{j-1}}{n-j+1}.
\end{aligned}
$$
Thus, we conclude that
\begin{equation}\label{II}
II=\sum_{j=1}^n\frac{(n+1)^{j-1}}{n-j+1}(n+1)^{1-j}=\sum_{j=1}^n\frac1j.
\end{equation}

We now turn to $I$.
The sum over $i$ that we  now need to consider is
\begin{equation}\label{A-1+1}
\begin{aligned}
&\sum_{i=1}^j\binom{j-1}{i-1}i^i(n-i+1)^{j-i-1}=\sum_{r=0}^{j-1}\binom{j-1}r(r+1)^{r+1}(n-r)^{(j-1)-r-1}=\\
&A_{j-1}(1,n-j+1;1,-1).
\end{aligned}
\end{equation}
In \cite{R}, a formula for $A_n(x,y;-1,1)$ ($=A_n(y,x;1,-1)$) is given, but  we were unable to exploit it for our purposes.
However, we were able to exploit the recursion  formula
$$
A_n(x,y;p,q)=A_{n-1}(x,y+1;p,q+1)+A_{n-1}(x+1,y;p+1,q),
$$
which can also be found in \cite{R}.
Using this recursion  formula, we have
\begin{equation}\label{recursionexploit}
A_{j-1}(1,n-j+1;1,-1)=A_j(0,n-j+1;0,-1)-A_{j-1}(0,n-j+2;0,0).
\end{equation}
By \eqref{Abelbinom},
\begin{equation}\label{0-1}
A_j(0,n-j+1;0,-1)=\frac{(n+1)^j}{n-j+1}.
\end{equation}
By \cite[page 23]{R},
$$
A_n(x,y;0,0)=\sum_{k=0}^n\binom nk(x+y+n)^k(n-k)!=n!\sum_{k=0}^n\frac{(x+y+n)^k}{k!}.
$$
Thus,
\begin{equation}\label{00}
A_{j-1}(0,n-j+2;0,0)=(j-1)!\sum_{k=0}^{j-1}\frac{(n+1)^k}{k!}.
\end{equation}
From \eqref{recursionexploit}-\eqref{00}, we obtain
\begin{equation}\label{1,-1}
A_{j-1}(1,n-j+1;1,-1)=\frac{(n+1)^j}{n-j+1}-(j-1)!\sum_{k=0}^{j-1}\frac{(n+1)^k}{k!}.
\end{equation}
Now
 \eqref{IandII}, \eqref{A-1+1} and  \eqref{1,-1} yield
\begin{equation}\label{I}
\begin{aligned}
&I=\sum_{j=1}^n A_{j-1}(1,n-j+1;1,-1)(n-j)(n+1)^{-j}=\\
&\sum_{j=1}^n\frac{n-j}{n-j+1}-\sum_{j=1}^n(n-j)(j-1)!\sum_{k=0}^{j-1}\frac{(n+1)^{k-j}}{k!}.
\end{aligned}
\end{equation}

From \eqref{II} and \eqref{I}, we have
\begin{equation}\label{I-II}
\begin{aligned}
&I+II=n-\sum_{j=1}^n(n-j)(j-1)!\sum_{k=0}^{j-1}\frac{(n+1)^{k-j}}{k!}=\\
&n-n\sum_{j=1}^n(j-1)!\sum_{k=0}^{j-1}\frac{(n+1)^{k-j}}{k!}+
\sum_{j=1}^nj!\sum_{k=0}^{j-1}\frac{(n+1)^{k-j}}{k!}.
\end{aligned}
\end{equation}
Using the well-known identity $\sum_{i=0}^r\binom{m+i}i=\binom{m+r+1}r$ in the third equality below, we
 have
\begin{equation}\label{oneterm}
\begin{aligned}
&\sum_{j=1}^n(j-1)!\sum_{k=0}^{j-1}\frac{(n+1)^{k-j}}{k!}=\sum_{l=1}^n(\frac1{n+1})^l\sum_{k=0}^{n-l}\frac1{k!}(k+l-1)!=\\
&\sum_{l=1}^n(\frac1{n+1})^l(l-1)!\sum_{k=0}^{n-l}\binom{k+l-1}k=\sum_{l=1}^n(\frac1{n+1})^l(l-1)!\binom nl=\\
&\sum_{l=1}^n\frac1l\frac{\prod_{i=0}^{l-1}(n-i)}{(n+1)^l}.
\end{aligned}
\end{equation}
Similarly, we have
\begin{equation}\label{otherterm}
\sum_{j=1}^nj!\sum_{k=0}^{j-1}\frac{(n+1)^{k-j}}{k!}=\sum_{l=1}^n(\frac1{n+1})^ll!\binom {n+1}{l+1}=\sum_{l=1}^n\frac{n+1}{l+1}\frac{\prod_{i=0}^{l-1}(n-i)}{(n+1)^l}.
\end{equation}
Using \eqref{oneterm} and \eqref{otherterm}, we can write the last two terms on the right hand side of \eqref{I-II} as
\begin{equation}\label{oneother}
\begin{aligned}
&-n\sum_{j=1}^n(j-1)!\sum_{k=0}^{j-1}\frac{(n+1)^{k-j}}{k!}+
\sum_{j=1}^nj!\sum_{k=0}^{j-1}\frac{(n+1)^{k-j}}{k!}=\\
&\sum_{j=1}^n\left(-\frac nl+\frac{n+1}{l+1}\right)\frac{\prod_{i=0}^{l-1}(n-i)}{(n+1)^l}=-\sum_{l=1}^n\frac{n-l}{l(l+1)}\frac{\prod_{i=0}^{l-1}(n-i)}{(n+1)^l}.
\end{aligned}
\end{equation}
From
\eqref{IandII}, \eqref{I-II} and    \eqref{oneother}, it follows that
the right hand side of \eqref{explrmax} is equal to the right hand side of \eqref{lrmaxformagain}.
\hfill $\square$

\section{Proof of part (ii) of Theorem \ref{lrmaxasympthm}}\label{lrmaxasympthmproofii}
In light of \eqref{lrmaxformagain}, we need to analyze
the asymptotic behavior of   $\sum_{l=1}^n\frac{n-l}{l(l+1)}\frac{\prod_{i=0}^{l-1}(n-i)}{(n+1)^l}$.
By Taylor's remainder theorem,
$$
-x-\frac12(\frac x{1-x})^2\le\log(1-x)\le-x,\ 0\le x<1.
$$
Writing $\log\frac{\prod_{i=0}^{l-1}(n-i)}{(n+1)^l}=\log\prod_{k=1}^l(1-\frac k{n+1})=\sum_{k=1}^l\log(1-\frac k{n+1})$, we then have
$$
  -\sum_{k=1}^l\left(\frac k{n+1}+\frac12\left(\frac k{n+1-k}\right)^2\right) \le \log\frac{\prod_{i=0}^{l-1}(n-i)}{(n+1)^l}\le -\sum_{k=1}^l\frac k{n+1},
$$
and thus
\begin{equation}\label{prodest}
e^{-\frac{l(l+1)}{2(n+1)}}e^{-\frac{l^3}{2(n+1-l)^2}}\le \frac{\prod_{i=0}^{l-1}(n-i)}{(n+1)^l}\le e^{-\frac{l(l+1)}{2(n+1)}},
\end{equation}
where we've used the fact that $\sum_{k=1}^l\frac12\left(\frac k{n+1-k}\right)^2\le \frac{l^3}{2(n+1-l)^2}$.
In what follows,  whenever we write that an expression depending on $n$ is $\theta^+(n^\alpha)$, we mean that it falls between $c_1n^\alpha$ and $c_2n^\alpha$ for all $n$, where
$c_1,c_2>0$.

We first prove the lower bound on $E_n^{\text{park}}X_n^{\text{LR-max}}$ in
\eqref{lrmasympform}. Thus, we need an upper bound on
$\sum_{l=1}^n\frac{n-l}{l(l+1)}\frac{\prod_{i=0}^{l-1}(n-i)}{(n+1)^l}$.
We break up the sum into two parts. In turns out that the optimal intermediate point
at which to break up the sum is $\theta^+(n^\frac12)$.  Thus, using \eqref{prodest} for the first inequality below, 
and recalling that $\sum_{l=1}^n\frac1{l(l+1)}=\sum_{l=1}^n\left(\frac1l-\frac1{l+1}\right)=1-\frac1{n+1}$,
we write
\begin{equation}\label{lowerboundasympexp}
\begin{aligned}
&\sum_{l=1}^n\frac{n-l}{l(l+1)}\frac{\prod_{i=0}^{l-1}(n-i)}{(n+1)^l}\le \sum_{l=1}^n\frac{n-l}{l(l+1)}e^{-\frac{l(l+1)}{2(n+1)}}=\\
&\sum_{l=1}^{[n^\frac12]}\frac{n-l}{l(l+1)}e^{-\frac{l(l+1)}{2(n+1)}}+\sum_{[n^\frac12]+1}^n\frac{n-l}{l(l+1)}e^{-\frac{l(l+1)}{2(n+1)}}\le
\sum_{n=1}^{[n^\frac12]}\frac{n-l}{l(l+1)}+\\
&e^{-\frac{n^2}{2(n+1)}}\sum_{[n^\frac12]+1}^n\frac{n-l}{l(l+1)}=\sum_{n=1}^n\frac{n-l}{l(l+1)}-
\left(1-e^{-\frac{n^2}{2(n+1)}}\right)\sum_{[n^\frac12]+1}^n\frac{n-l}{l(l+1)}=\\
&n(1-\frac1{n+1})-\sum_{l=1}^n\frac1{l+1}-\theta^+(n^\frac12)=n-\theta^+(n^\frac12).
\end{aligned}
\end{equation}
The lower bound in \eqref{lrmasympform} now follows from \eqref{lrmaxformagain} and \eqref{lowerboundasympexp}.

We now turn to the proof of the upper bound in \eqref{lrmasympform}. Thus,  now we need a lower bound on
$\sum_{l=1}^n\frac{n-l}{l(l+1)}\frac{\prod_{i=0}^{l-1}(n-i)}{(n+1)^l}$.
Note that for $l=l_n=o(n)$, the exponent on the left hand side of \eqref{pr1odest} is increasing in $l$ and satisfies
$\frac{l_n(l_n+1)}{2(n+1)}+\frac{l^3}{2(n+1-l)^2}=\theta^+(\frac{l_n^2}n)$.
Thus, for any $r\in\mathbb{N}$ and $\{\alpha_i\}_{i=1}^r$ satisfying
$a_1<\alpha_2<\cdots<\alpha_r<1$, we have from \eqref{prodest}
\begin{equation}\label{lowerbd}
\begin{aligned}
&\sum_{l=1}^n\frac{n-l}{l(l+1)}\frac{\prod_{i=0}^{l-1}(n-i)}{(n+1)^l}\ge\\
& e^{-\theta^+\left(n^{2\alpha_1-1}\right)}\sum_{l=1}^{[n^{\alpha_1}]}\frac{n-l}{l(l+1)}+
\sum_{i=2}^re^{-\theta^+\left(n^{2\alpha_i-1}\right)}\sum_{l=[n^{\alpha_{i-1}}]+1}^{[n^{\alpha_i}]}\frac{n-l}{l(l+1)}\ge\\
&\left(1-\theta^+\left(n^{2\alpha_1-1}\right)\right)\left(n(1-\frac1{[n^{\alpha_1}]+1})-\sum_{l=1}^{[n^{\alpha_1}]}\frac1{l+1}\right)+\\
&\sum_{i=2}^r\left(1-\theta^+\left(n^{2\alpha_i-1}\right)\right)
\left(n(\frac1{[n^{\alpha_{i-1}}]+1}-\frac1{[n^{\alpha_i}]+1})-
\sum_{l=[n^{\alpha_{i-1}}]+1}^{[n^{\alpha_i}]}\frac1{l+1}\right).
\end{aligned}
\end{equation}
Note that
\begin{equation}\label{lotsofthetas}
\begin{aligned}
&\left(1-\theta^+\left(n^{2\alpha_1-1}\right)\right)\left(n(1-\frac1{[n^{\alpha_1}]+1})-\sum_{l=1}^{[n^{\alpha_1}]}\frac1{l+1}\right)=\\
&n(1-\frac1{[n^{\alpha_1}]+1})-\theta^+\left(n^{2\alpha_1}\right);\\
&\left(1-\theta^+\left(n^{2\alpha_i-1}\right)\right)
\left(n(\frac1{[n^{\alpha_{i-1}}]+1}-\frac1{[n^{\alpha_i}]+1})-
\sum_{l=[n^{\alpha_{i-1}}]+1}^{[n^{\alpha_i}]}\frac1{l+1}\right)=\\
&n(\frac1{[n^{\alpha_{i-1}}]+1}-\frac1{[n^{\alpha_i}]+1})-\theta^+\left(n^{2\alpha_i-\alpha_{i-1}}\right),\ i=2,\cdots, r.
\end{aligned}
\end{equation}
From \eqref{lotsofthetas} and \eqref{lowerbd}, we obtain
\begin{equation}\label{n-thetas}
\begin{aligned}
&\sum_{l=1}^n\frac{n-l}{l(l+1)}\frac{\prod_{i=0}^{l-1}(n-i)}{(n+1)^l}\ge\\
&n-\theta^+\left(n^{1-\alpha_r}\right)-\theta^+\left(n^{2\alpha_1}\right)-\sum_{i=2}^r\theta^+\left(n^{2\alpha_i-\alpha_{i-1}}\right).
\end{aligned}
\end{equation}
We now set $\{\alpha_i\}_{i=1}^r$ in order to minimize the order of
$\theta^+\left(n^{1-\alpha_r}\right)+\theta^+\left(n^{2\alpha_1}\right)+\sum_{i=2}^r\theta^+\left(n^{2\alpha_i-\alpha_{i-1}}\right)$.
The minimal will occur when all of the exponents $1-\alpha_r,2\alpha_1,\{2\alpha_i-\alpha_{i-1}\}_{i=2}^r$ are equal to one another.
Solving these equations, one obtains
$$
\alpha_j=\frac{2^j-1}{2^{j-1}}\frac{2^{r-1}}{2^{r+1}-1},\ j=1,\cdots, r.
$$
The common value of the exponents is then $\frac{2^r}{2^{r+1}-1}$
Therefore, from \eqref{lrmaxformagain} and \eqref{n-thetas}, we conclude
that 
$$
E_n^{\text{park}}X_n^{\text{LR-max}}\le \theta^+\left(n^{\frac{2^r}{2^{r+1}-1}} \right).
$$ 
Since $r\in\mathbb{N}$ is arbitrary, this proves the upper bound in \eqref{lrmasympform}.
\hfill $\square$


\begin{thebibliography}{99}
\bibitem{B} Bona, M., Combinatorics of Permutations, CRC Press, Boca Raton, Fl. (2012).

\bibitem{DH}  Diaconis, P. and Hicks, A.,   \emph{ Probabilizing parking functions}, Adv. in Appl. Math. \textbf{89} (2017), 125–155.

\bibitem{EH}
Ehrenborg, R. and Happ, A.,
\emph{Parking cars after a trailer}, Australas. J. Combin. \textbf{70} (2018), 402–406.

\bibitem{FR}
Foata, D, and Riordan, J., \emph{Mappings of acyclic and parking functions}, Aequ. Math. \textbf{10} (1974), 10-22.


\bibitem{P14} Pinsky, R.G. \emph{Problems From the Discrete to the Continuous},
Universitext
Springer, Cham, (2014).

\bibitem{R} Riordan, J., \emph{Combinatorial Identities}, John Wiley \& Sons, Inc., New York-London-Sydney, (1968).

\bibitem{Y} Yan, C.H., \emph{Parking functions},
Discrete Math. Appl.
CRC Press, Boca Raton, FL, (2015), 835–893.
\end{thebibliography}
\end{document}